\documentclass[12pt]{article}
\usepackage{amssymb,amsmath}

\newtheorem{thm}     {Theorem}[section]
\newtheorem{prop}    [thm]{Proposition}

\newtheorem{cor}     [thm]{Corollary}
\newtheorem{lemma}   [thm]{Lemma}

\newcommand{\proof} {\noindent{\bf Proof. }}

\newcommand{\B}{\mathbb B}
\newcommand{\C}{\mathbb C}
\newcommand{\D}{\mathbb D}
\def\H{\mathbb H}
\def\P{\mathbb P}
\newcommand{\R}{\mathbb R}

\newcommand{\Z}{\mathbb Z}

\def\Re{{\rm Re\,}}
\def\Im{{\rm Im\,}}
\def\bar{\overline}
\def\Area{{\rm Area}}
\def\<{\langle}
\def\>{\rangle}
\def\Span{{\rm Span}}

\begin{document}

\title{Symplectic non-squeezing in Hilbert space\\
and discrete Schr\"odinger equations}
\author{Alexandre Sukhov{*} and Alexander Tumanov{**}}
\date{}
\maketitle

{\small
* Universit\'e des Sciences et Technologies de Lille, Laboratoire
Paul Painlev\'e, U.F.R. de Math\'e-matique, 59655 Villeneuve d'Ascq, Cedex, France. The author is partially supported by Labex CEMPI.
E-mail address: sukhov@math.univ-lille1.fr

** University of Illinois, Department of Mathematics, 1409 West Green Street, Urbana, IL 61801, USA.
The author is partially supported by Simons Foundation grant.
E-mail address: tumanov@illinois.edu
}
\bigskip

Abstract. We prove a generalization of Gromov's symplectic non-squeezing theorem for the case of Hilbert spaces. Our approach is based on filling almost complex Hilbert spaces by complex discs partially extending Gromov's results on existence of $J$-complex curves. We apply our result to the flow of the discrete nonlinear Schr\"odinger equation.
\bigskip



\tableofcontents

\section{Introduction}

In the space $\R^{2n}$ with coordinates $(x_1,...,x_n,y_1,...,y_n)$ and standard symplectic form $\omega = \sum_j dx_j \wedge dy_j$, we consider the Euclidean unit ball $\B$ and the cylinder $\Sigma = \{ (x,y): x_1^2 + y_1^2 < 1 \}$. Gromov's non-squeezing theorem \cite{Gr} states that if for some $r,R > 0$ there exists a symplectic  embedding $f: r\B \to R\Sigma$, that is,  $f^*\omega = \omega$, then $r \le R$.
This result had a deep impact on the development of the symplectic geometry. In contrast to the case of finite-dimensional symplectic manifolds arising from the classical mechanics and dynamics, the symplectic structures and flows corresponding Hamiltonian PDEs are defined on suitable Hilbert spaces, usually Sobolev spaces (see for instance \cite{Ku2}). This explains the interest in analogs of Gromov's theorem for symplectic Hilbert spaces. The first non-squeezing result for symplectic flows of  various classes of Hamiltonian PDEs was obtained by Kuksin  \cite{Ku1} and later extended in the work of Bourgain  \cite{Bou1,Bou2}, Colliander, Keel, Staffilani, Takaoka, and Tao \cite{Tao}, Roum\'egoux \cite {Rou}. Their approach is based on approximation of a symplectic flow on a Hilbert space by finite-dimensional symplectic flows which reduces the situation to Gromov's theorem. It seems natural to look for a general analog of Gromov's theorem for symplectic Hilbert spaces.
Abbondandolo and Majer  \cite{Ab} prove the result in the case where the symplectic image $f(r\B)$ of the Hilbert ball $r\B$ is convex. Finally, Fabert \cite{F} has recently proposed a proof of the result for general symplectic flows in Hilbert spaces using non-standard analysis.

In the present work we prove a generalization of Gromov's non-squeezing theorem to the case of symplectic Hilbert spaces under assumptions of boundedness and regularity of the symplectic transformation in certain Hilbert scales. Gromov's original proof uses almost complex structures $J$ tamed by the standard symplectic form on the complex projective space $\C\P^n$; the key technical tool is filling the projective space by $J$-complex spheres. An immediate attempt to extend this construction to the case of Hilbert spaces leads to difficulties because the main ingredients of Gromov's theory (compactness and transversality for $J$-complex curves) are not available. We use the method introduced in our previous paper \cite{SuTu1} where we give a new simple proof of Gromov's theorem.  This approach can be extended with suitable modifications to the Hilbert space case. The main idea is to replace $J$-complex spheres in Gromov's argument by $J$-complex discs with boundaries attached to the boundary of a cylinder. These discs are Hilbert space valued functions satisfying a certain first order quasilinear system of PDE with non-linear boundary conditions. The integral equation corresponding to this boundary value problem has a solution by combination of the contraction mapping principle and the Schauder fixed point theorem. Similar methods are known in the theory of the scalar Beltrami equation which partially inspired our approach; we extend them to vector valued functions. In the last section, we apply our main result to the flows of infinite systems of ODEs, in particular, discrete nonlinear Schr\"odinger equations.

The authors wish to thank Marius Junge and Zhong-Jin Ruan for their help with vector-valued $L^p$ spaces and Stephan de Bievre for useful discussions. We are grateful to the referee for suggesting to consider the discrete Schr\"odinger equation.

\section{Almost complex structures on Hilbert spaces}

In this section we introduce almost complex structures (see \cite{Aud}), spaces of vector-valued functions, and Hilbert scales (see \cite{Ku2}). We include some auxiliary results concerning almost complex structures in Hilbert spaces because we could not find precise references.

\subsection{Almost complex and  symplectic structures}
Let $V$ be a real vector space. If $V$ has finite dimension, then we assume that $\dim V$ is even.  A {\it linear almost complex structure} $J$ on $V$ is a bounded linear operator $J: V \to V$ satisfying $J^2 = -I$. Here and below  $I$ denotes the identity map or the identity matrix depending on the context.

Let $\H$ be a complex Hilbert space with Hermitian  scalar product $\langle \bullet, \bullet \rangle$; we consider only {\it separable} Hilbert spaces. Fix an orthonormal basis $\{ e_j \}_{j=1}^\infty$ of
$\H$ such that $Z = \sum_{j=1}^\infty Z_j e_j$ for every $Z \in \H$. Here $Z_j = x_j + iy_j= \langle Z,e_j \rangle$ are complex coordinates of $Z$. Then $\H$ can be identified with the complex space $l^2$. We will use complex conjugation $\bar{Z} = \sum_{j=1}^\infty \bar{Z}_j e_j$. The {\it standard almost complex structure} $J_{st}$ on $\H$ is the operator defined as $J_{st}Z=iZ$, hence, $J_{st}^2 = -I$. In the case where $\H$ has a finite dimension $n$ the structure $J_{st}$ is the usual complex structure on $\C^n$. We do not specify the dimension (finite or infinite) in this notation since it will be clear from the context. Denote by ${\mathcal L}(\H)$  the space  of real linear bounded operators on $\H$. An {\it almost complex structure} $J$ on $\H$ is a continuous map   $J: \H \to {\mathcal L}(\H)$,   $J: \H \ni Z \to J(Z)$ such that every $J(Z)$  satisfies $J^2(Z) = -I$. If the map $Z \mapsto J(Z)$ is independent of $Z$, then we can identify the tangent space of $\H$ at $Z$ with $\H$ and view $J$ as a linear almost complex structure on $\H$.

Denote by  $\D = \{ \zeta \in \C : |\zeta| < 1 \}$  the unit disc in $\C$. It is equipped with the standard complex structure $J_{st}$ of $\C$. Let $J$ be an almost complex structure on $\H$. A $C^1$- map
$f:\D \to \H$ is called  a $J$-{\it complex disc} in $\H$ if it satisfies {\it the Cauchy-Riemann equations}
\begin{eqnarray}
\label{CRglobal}
J \circ df= df \circ J _{st}.
\end{eqnarray}
We use  the complex derivatives
\[
f_Z = \frac{\partial f}{\partial Z} = \frac{1}{2} \left ( \frac{\partial f}{\partial x} - i\frac{\partial f}{\partial y} \right), \,\,\,\,\,\,  f_{\bar Z} = \frac{\partial f}{\partial \bar Z}=\frac{1}{2}\left ( \frac{\partial f}{\partial x} + i\frac{\partial f}{\partial y} \right ),
\]
where $f$ is a map between two (finite or infinite dimensional) Hilbert spaces. It is convenient to rewrite (\ref{CRglobal}) in complex notation.
Assume that  for all $Z \in \H$ the operator
$(J_{st} + J)(Z)$ is invertible. Then the linear operator
\[
L:= (J_{st} + J)^{-1}(J_{st} - J)
\]
is well defined. Like in the finite-dimensional case (see \cite{Aud}), the operator $L$ is $J_{st}$-anti-linear, i.e.,
$J_{st} L = -L J_{st}$. Hence, there exists a bounded
$J_{st}$-linear operator $A_J: \H \to \H$ such that
\[
L h = A_J \bar{h}.
\]
We call $A_J$ the {\it complex representation} of $J$ and often omit $J$.  With this convention the Cauchy-Riemann equations (\ref{CRglobal}) for a $J$-complex disc $Z:\D\to\H$, $Z: \D \ni \zeta \mapsto Z(\zeta)$ can be  written  in  the form
\begin{eqnarray}
\label{holomorphy}
Z_{\bar\zeta}=A_J(Z)\bar Z_{\bar\zeta},\quad
\zeta\in\D.
\end{eqnarray}

The {\it standard symplectic form } $\omega$ on $\H$ is  a nondegenerate antisymmetric bilinear form defined by
\[
\omega = \frac{i}{2}\sum_{j=1}^\infty dZ_j \wedge d\bar{Z}_j.
\]
We  use the natural identification of $\H$ with its tangent space at every point.

For a map $Z: \D \to \H$, $Z: \zeta \mapsto Z(\zeta)$ its (symplectic) {\it area}  is defined by
\begin{eqnarray}
\label{area}
\Area(Z) = \int_\D Z^*\omega
\end{eqnarray}
similarly to the finite-dimensional case. If $Z$ is $J_{st}$-holomorphic,  then (\ref{area}) represents its  area induced by the inner product of $\H$.

\subsection{Symplectomorphisms}

Let $\H$ be a complex Hilbert space with fixed basis and the standard
symplectic form $\omega$.
By default, all linear operators are bounded. For an $\R$-linear operator $F: \H \to \H$ we denote by $F^*$  its adjoint, that is $\Re \langle Fu,v \rangle = \Re \langle u, F^* v \rangle$. Put
\begin{eqnarray*}
\bar{F} Z = \bar{(F \bar Z)}, \,\,\, \mbox{and} \,\,\, F^{t} = \bar{ F^*}.
\end{eqnarray*}
Thus $F^t$ is the transpose of $F$.

Every $\R$-linear operator $F: \H \to \H$ can be uniquely written in the form
\[
F u = P u + Q \bar{u},
\]
where $P$ and $Q$ are $\C$-linear operators. For brevity we write
\[
F = \{ P, Q\}.
\]
Note that
\begin{eqnarray*}
F^* = \{ P^*, Q^t \}, \,\,\, F^t = \{ P^t, Q^* \}.
\end{eqnarray*}
The following two lemmas are proved in \cite{SuTu5}.

\begin{lemma}
\label{LemAC1}
Let $F= \{ P, Q \}$. Then $F$ preserves $\omega$, i.e.,
$\omega(Fu,Fv) = \omega(u,v)$ if and only if
\begin{eqnarray}
\label{identity1}
P^* P - Q^t \bar{Q} = I
\,\,\,\mbox{and}\,\,\,
P^t \bar{Q} - \bar{Q}^t P = 0.
\end{eqnarray}
\end{lemma}

A linear operator $F:\H \to \H$
is called a {\it linear symplectomorphism}
if $F$ is invertible and preserves $\omega$.

\begin{lemma}
\label{LemAC2}
Let $F = \{ P, Q \}$ be a linear symplectomorphism. Then $F^t$ also preserves $\omega$, that is,
\begin{eqnarray}
\label{identity2}
P P^* - Q Q^* = I\,\,\,\,\mbox{and}\,\,\,P Q^t - Q P^t = 0.
\end{eqnarray}
\end{lemma}

\begin{prop}
\label{PropAC3}
Let $F = \{ P, Q \}$ be a linear symplectomorphism. Then
\begin{itemize}
\item[(a)] $ F^{-1} = \{ P^*, -Q^t \}$;
\item[(b)] $P$ is invertible;
\item[(c)] $\| Q \bar{P}\,^{-1} \| = \| Q \| ( 1 + \| Q \| ^2)^{-1/2} < 1$.
\end{itemize}
\end{prop}
\begin{proof} For convenience, we include the proof from \cite{SuTu5}. Part (a) follows by (\ref{identity1}) and (\ref{identity2}). By (\ref{identity1}) and (\ref{identity2}), spectral values of the self-adjoint operators $P P^*$ and $P^* P$ are not smaller that $1$. Then both $P^* P$ and $PP^*$ are invertible which gives (b). For (c), put $A = Q \bar{P}\,^{-1}$. We estimate $\| A \| = \| A A^* \|^{1/2}$.
By (\ref{identity1}) and (\ref{identity2}) respectively, we have $Q \bar{P}\,^{-1} =\bar{P^t}\,^{-1}Q^t$ and $Q^t (P^t)^{-1} = P^{-1}Q$. Using the latter, $A A^* = (P P^*)^{-1} Q Q^*$. Since $P P^* = I + Q Q^*$ and $Q Q^*$ is self-adjoint, by the spectral mapping theorem
\[
\| A A^* \| = \frac{\| Q Q^* \|}{1 + \| Q Q^* \|} = \frac{\| Q \|^2}{1 + \| Q \|^2}
\]
because the function $\lambda \mapsto \lambda(1+ \lambda)^{-1}$  is increasing for $\lambda>0$.
\end{proof} $\blacksquare$

A $C^1$-diffeomorphism (continuously Fr\'echet differentiable map) $\Phi: \Omega_1 \to \Omega_2$ between two open subsets $\Omega_j$ in $(\H,\omega)$ is called a {\it symplectomorphism} if $\Phi^*\omega = \omega$. Here the star denotes the pull-back. In the proof of one of our main results (Theorem \ref{squiz}), we encounter an almost complex structure $J$ arising as the direct image
\[
J = \Phi_*(J_{st}):= d\Phi \circ J_{st} \circ d\Phi^{-1}
\]
of $J_{st}$ under a symplectomorphism $\Phi: (\H,\omega) \to (\H,\omega)$. We claim that such almost complex structure $J$ has a complex representation $A_J$, so the Cauchy-Riemann equations for $J$ can be written in the form (\ref{holomorphy}).

\begin{lemma}
\label{MatrixA}
Let $\Phi:\Omega_1 \to \Omega_2$ be a diffeomorphism of class $C^1$ between two open subsets $\Omega_j$, $j=1,2$ of $\H$. Put $P= \Phi_{Z}$ and $Q= \Phi_{\bar Z}$.
Then the complex representation $A_J$ of the direct image $J = \Phi_*(J_{st})$ has the form
\begin{eqnarray}
\label{MA}
A_J = Q\bar{P}\,^{-1}
\end{eqnarray}
\end{lemma}
Indeed, by Proposition \ref{PropAC3} the operator $P$ is invertible for all $Z$.
Then the conclusion follows by Lemma 2.3 from \cite{SuTu2}, whose proof  goes through for the Hilbert space case without changes.

\subsection{Hilbert scales}

Let $\H$ be a complex Hilbert space with fixed basis.
Let $(\theta_n)_{n=1}^\infty$ be a sequence of positive numbers
such that $\theta_n \to \infty$ as $n \to \infty$,
for example, $\theta_n=n$.
Introduce a diagonal operator
\[
D={{\rm Diag}}(\theta_1, \theta_2, \ldots).
\]
For $s\in\R$ we define $\H_s$ as a Hilbert space with the following
inner product and norm:
\[
\<x,y\>_s=\<D^s x,D^s y\>, \quad
\|x\|_s=\|D^s x\|.
\]
Thus $\H_0=\H$,
$\H_s=\{x\in \H: \|x\|_s<\infty \}$ for $s>0$, and
$\H_s$ is the completion of $\H$ in the above norm for $s<0$.
The family $(\H_s)$ is called a Hilbert scale corresponding
to the sequence $(\theta_n)$. For $s > r$, the space $\H_s$
is dense in $\H_r$, and the inclusion $\H_s \subset \H_r$ is compact.
We refer to \cite{Ku2} for details.

We need a version of Proposition \ref{PropAC3} for Hilbert scales.
\begin{prop}
\label{PropHS}
Let $(\H_s)$ be a Hilbert scale.
Let $F = \{ P, Q \}$ be a linear symplectomorphism of the standard
symplectic structure on $\H=\H_0$.
Let $s_0, C>0$ be constants such that
$\|F\|_s\le C$ and $\|F^{-1}\|_s\le C$ for $0\le s\le s_0$.
Then there exist constants $s_1>0$ and $0<a<1$ depending only
on $s_0$ and $C$ such that for $0\le s\le s_1$
\begin{itemize}
\item[(a)] $ \|P^{-1}\|_s \le 2C$;
\item[(b)] $\| Q \bar P\,^{-1} \|_s \le a$.
\end{itemize}
\end{prop}

\begin{lemma}
\label{lemmaHS}
Let $Q$ be a linear operator in $\H_0$.
Suppose $\|Q\|_s\le C$ for real $|s|\le s_0$.
Then $\|D^s Q D^{-s}-Q\|_0\le 2Cs_0^{-1}|s|$ for $|s|\le s_0$.
\end{lemma}
\begin{proof}
One easily verifies $\|Q\|_s=\|D^s Q D^{-s}\|_0$.
Introduce $f(s)=D^s Q D^{-s}-Q$ as a holomorphic
function of complex variable $s$.
Since for $t\in \R$ the operator $D^{it}$ is unitary,
we have $\|f(s)\|_0\le 2C$ in the strip $|\Re s|\le s_0$,
in particular, in the disc $|s|\le s_0$.
Since $f(0)=0$, by the Schwarz lemma we get the
desired estimate.
\end{proof} $\blacksquare$
\medskip

\noindent
{\bf Proof of Proposition \ref{PropHS}.}
Since $Pz=\frac{1}{2}(Fz-iF(iz))$ and
$Qz=\frac{1}{2}(F\bar z+iF(i\bar z))$, we have
$\|P\|_s\le C$ and $\|Q\|_s\le C$.
By Proposition \ref{PropAC3} (a) we also have
$\|P^*\|_s\le C$ and $\|Q^*\|_s\le C$, here the stars
stand for the adjoints in $\H_0$.
By \eqref{identity2}, $P^{-1}=P^*(I+QQ^*)^{-1}$,
which we will use to estimate $P^{-1}$.

Since $\|Q\|_s\le C$ for $0\le s\le s_0$,
we have $\|D^s Q D^{-s}\|_0\le C$ for such $s$.
Passing to the adjoint, $\|D^{-s} Q^* D^s\|_0\le C$ for
$0\le s\le s_0$.
Since $\|Q^*\|_s\le C$ also, we have
$\|D^{s} Q^* D^{-s}\|_0\le C$ for $0\le s\le s_0$.
Altogether,
$\|D^s Q^* D^{-s}\|_0\le C$ for all $|s|\le s_0$.
By Lemma \ref{lemmaHS},
$\|D^s Q^* D^{-s}-Q^*\|_0\le 2Cs_0^{-1}|s|$ for $|s|\le s_0$.

Denote by $Q^*_s$ the adjoint of $Q$ in $\H_s$.
One easily verifies $Q^*_s=D^{-2s} Q^* D^{2s}$.
We rewrite $(I+QQ^*)^{-1}=(I+K+L)^{-1}$, here
$K=QQ^*_s\ge 0$ in $\H_s$ and $L=Q(Q^*-Q^*_s)$.
We claim that $L$ is small for small $s$. Indeed,
\[
\|Q^*-Q^*_s\|_s=\|D^s(Q^*-D^{-2s} Q^* D^{2s})D^{-s}\|_0
\le \|D^s Q^* D^{-s}-Q^*\|_0
+\|D^{-s} Q^* D^s-Q^*\|_0
\le 4Cs_0^{-1}s
\]
for $0\le s\le s_0$.
Then $\|L\|_s\le 4C^2 s_0^{-1}s$ for such $s$.
Put $s_1=\frac{1}{8}C^{-2}s_0$. Then $\|L\|_s\le\frac{1}{2}$
for  $0\le s\le s_1$.

Since $K\ge 0$ in $\H_s$, we have $\|(I+K)^{-1}\|_s\le 1$, moreover,
\[
\|(I+QQ^*)^{-1}\|_s = \|(I+(I+K)^{-1}L)^{-1}(I+K)^{-1}\|_s\le 2.
\]
Hence $\|P^{-1}\|_s\le 2C$ for $0\le s\le s_1$ as stated in (a).

We now estimate $A=Q\bar P\,^{-1}$.
By Proposition \ref{PropAC3},
$\|A\|_0\le a_0:= C(1+C^2)^{1/2}<1$.
By part (a), $\|A\|_s\le a_1:=2C^2$, $0\le s\le s_1$.
By interpolation (see \cite{Ku2}) for $0\le s\le s_1$,
\[
\|A\|_s\le a_0^{1-(s/s_1)}a_1^{s/s_1}\to a_0
\quad {{\rm as}} \quad s\to 0.
\]
Let $a_0<a<1$, say $a=\frac{1+a_0}{2}$. Now by shrinking $s_1$
if necessary we obtain $\|A\|_s\le a$ for all $0\le s\le s_1$,
as desired.
$\blacksquare$

\subsection{Vector-valued Sobolev spaces}

Let $X$ be a Banach space. Denote by   $W^{k,p}(\D,X)$  the  Sobolev classes of maps $Z:\D \to X$ admitting weak partial derivatives $D^\alpha Z\in L^p(\D,X)$ up  to the order $k$ (as usual we identify functions coinciding almost everywhere).  We define weak derivatives in the usual way using the space $C^\infty_0(\D)$ of smooth scalar-valued test functions with compact support in $\D$. Then $W^{k,p}(\D,X)=L^p(\D,X)$ if $k=0$. The norm on $L^{p}(\D,X)$ is defined by
\[
\| Z \|_{L^p(\D,X)} = \left ( \int_{\D} \| Z(\zeta) \|_X^p d^2\zeta \right )^{1/p}.
\]
Here $d^2\zeta:= (i/2)d\zeta \wedge d\bar{\zeta}$ denotes the standard Lebesgue measure in $\R^2$.
The space $W^{k,p}(\D,X)$ equipped with the norm
\[
\| Z \| = \left ( \sum_{ |\alpha| \leq k} \| D^\alpha Z \|^p_{L^p(\D,X)}\right )^{1/p}
\]
is a Banach space.
We use the standard notation  $C^{\alpha}(\D,X)$ for the Lipschitz space. Denote also by $C(\bar\D,X)$ the space of vector functions continuous on $\bar\D$ equipped with the sup-norm. We will deal with the case $X = \H_s$.
We will need the following analog of Sobolev's compactness theorem.
\begin{prop}
\label{SobolevProp}
The inclusion
\begin{equation}
\label{Sobolev}
W^{1,p}(\D,\H_r) \subset C(\bar\D,\H_s),\;
s<r,\; p>2
\end{equation}
is compact.
\end{prop}
This result is well-known \cite{Aub} in the case of vector functions on an interval of $\R$.

\begin{proof}
We decompose \eqref{Sobolev} into
\begin{equation}
\label{Sobolev1}
W^{1,p}(\D,\H_r) \subset C^{\alpha}(\D,\H_r) \subset
C(\bar\D,\H_s).
\end{equation}
 The first inclusion in \eqref{Sobolev1}
is Morrey's embedding with $\alpha = (p-2)/p$ (see, for example, \cite{SuTu5}). The second inclusion is compact by the Arzela-Ascoli theorem, hence \eqref{Sobolev} is compact.
\end{proof} $\blacksquare$

Sobolev's compactness from Proposition \ref{SobolevProp} plays an important role in our argument. It replaces in some sense Gromov's compactness for pseudo-holomorphic curves. This is
the main reason we use Hilbert scales.
We finally note that the system (\ref{holomorphy}) still makes sense for $Z \in W^{1,p}(\D)$ with $p> 2$.

\section{Main results}

Let $\H$ be a complex Hilbert space with fixed orthonormal basis
and the standard symplectic form $\omega$. Let
$(\H_s)$ be a Hilbert scale, so that $\H_0=\H$. Denote by
\[
\B^\infty = \{ Z \in \H: \| Z \| < 1 \}
\]
the unit ball in $\H_0$. Then $r\B^\infty$ is the ball of radius $r > 0$.
We now use the notation
\[
Z = (z,w) = (z,w_1,w_2,\ldots)
\]
for the coordinates in $\H_0$.
Here $z= \langle Z,e_1 \rangle\in\C$. For a domain
$\Omega \subset \C$ we define  the cylinder
\[
\Sigma_\Omega= \{ Z \in \H_0: z \in \Omega \}
\]
in $\H_0$. Our first main result is the following

\begin{thm}
\label{squiz}
(Non-squeezing theorem.) Let $r,R > 0$ and let  $G$ be an open subset in $\Sigma_{R\D}$. Suppose that there exists a symplectomorphism
$\Phi: r\B^\infty \to G$
of class $C^1$ with respect to the $\H_0$-norm. Let $s_0>0$. Suppose that for every $s \in [0,s_0]$ the tangent maps $d\Phi(Z): \H_0 \to \H_0$ as well as their inverses $(d\Phi(Z))^{-1}$ are in fact bounded in the $\H_s$ norm (as operators  $\H_s \to \H_s$) uniformly in $Z\in r\B^\infty$. Then $r \leq R$.
\end{thm}

Obviously, the scale regularity  assumption  holds in the finite-dimensional case.  Hence, Theorem \ref{squiz}
generalizes Gromov's theorem.

In view of interpolation theorems for linear operators in Hilbert scales (see, for example, \cite{Ku2}), it suffices to assume that the tangent maps for $\Phi$ are uniformly bounded for $s=0$ and
$s = s_0$.

We also note that by Proposition \ref{PropAC3} the boundedness of the inverses $(d\Phi(Z))^{-1}$
automatically follows from the boundedness of $d\Phi(Z)$ for $s=0$,
but not for $s>0$. Instead of the assumption on the inverses,
we can assume that the tangent maps $d\Phi(Z)$ are uniformly bounded in $\H_s$ for $|s|\le s_0$, that is, also for negative $s$.

It may seem reasonable to require the boundedness of $d\Phi(Z)$
only for $Z\in\H_s$. However, by the principle of uniform
boundedness and continuity of $d\Phi$, it would imply uniform
boundedness on the whole ball, hence our hypothesis does not
restrict generality.

This is clear that Theorem \ref{squiz} can be applied to symplectomorphisms between  an
arbitrary ball  (not necessarily centered at the origin) and  a cylinder obtained from $\Sigma_{R\D}$ by an affine translation and a permutation of coordinates because such transformations are symplectic.

Theorem \ref{squiz} is a consequence of our second main result on the existence of $J$-complex discs.
Following \cite{SuTu1}, we replace a circular cylinder by a triangular one. The reason is that the construction of $J$-complex discs in a circular cylinder leads to a boundary value problem for the Cauchy-Riemann equations with non-linear boundary conditions.
For the triangular cylinder, the boundary conditions become linear although with discontinuous coefficients. The latter can be handled
by means of modified Cauchy-Green operators \cite{AM}.

Denote by $\Delta$ the triangle
$\Delta = \{ z \in \C: 0 < \Im z < 1 - |\Re z| \}$.
Note that   $\Area(\Delta) = 1$.

\begin{thm}
\label{ThDiscs}
Let $\Sigma = \Sigma_\Delta$.
Let $A(Z): \H_0 \to \H_0$, be a  family of  bounded linear  operators continuously depending on $Z \in \H_0$ and such that $A(Z) = 0$ for  $Z \in\H_0\setminus\Sigma$. Suppose that $A(Z):\H_s \to \H_s$ is bounded for every $s\in [0,s_0]$. Furthermore, suppose there is $a < 1$ so that for all
$Z \in \H_0$
\begin{eqnarray}
\label{norm2}
\| A(Z) \|_{\H_s} \leq a.
\end{eqnarray}
Then there exists $p > 2$ such that for every point $(z^0,w^0) \in \Sigma$ there is a solution $Z\in W^{1,p}(\D,\H_0)$ of (\ref{holomorphy}) such that $Z(\bar\D)\subset\bar\Sigma$, $(z^0,w^0) \in Z(\D)$, $\Area(Z) = 1$, and
\begin{eqnarray}
\label{BC0}
Z(b\D) \subset
b\Sigma.
\end{eqnarray}
\end{thm}

In reducing Theorem \ref{squiz} to Theorem \ref{ThDiscs}, we essentially follow Gromov's \cite{Gr} argument.
\medskip

\noindent
{\bf Proof of Theorem \ref{squiz}.}
A diffeomorphism whose $z$-component is an area-preserving map and whose $w$-components are the identity maps, preserves the form $\omega$. This observation reduces the proof to the case where $G$ is contained in the triangular cylinder
$\Sigma: = \{ (z,w): z \in \sqrt{\pi}R\Delta \}$.

Put $\tilde J = \Phi_*(J_{st})$. Put $P = \Phi_Z$ and $Q = \Phi_{\bar Z}$. By Lemma \ref{MatrixA}  the complex representation $\tilde A$
of $\tilde J$ has the form (\ref{MA}). Furthermore, by Proposition \ref{PropHS} (b), $\tilde A(Z)$ satisfies (\ref{norm2}) for all $Z \in G$ and $s \in [0,s_0]$.

Fix $\varepsilon > 0$.
Let $0\le\chi\le 1$ be a smooth cut-off function with support in $G$
and such that $\chi=1$ on $\Phi((r-\varepsilon)\bar\B^\infty)$.
Define $A=\chi\tilde A$.
Since $\chi\le 1$, the estimate (\ref{norm2}) holds for $A$.

Let $p = \Phi(0)$.
By Theorem \ref{ThDiscs} there exists a solution $Z$ of (\ref{holomorphy}) such that $p \in Z(\D)$, $Z(b\D) \subset b\Sigma$ and $\Area(Z) = \pi R^2$.
Denote by $D \subset \D$ a connected component of the pre-image $Z^{-1}(\Phi((r-\varepsilon)\B^\infty))$.
Then $X = \Phi^{-1}(Z(D))$ is a closed $J_{st}$-complex curve  in $(r-\varepsilon)\B^\infty$
with boundary contained in $(r-\varepsilon)b\B^\infty$. Furthermore, $0 \in X$ and $\Area(X) \leq \pi R^2$.

Consider the canonical projection $\pi_n: \H \to \C^n$, $\pi_n: z = (z_1, z_2 , ...) \mapsto (z_1, z_2, ... , z_n)$. Set $Z' = \Phi^{-1} \circ Z$.
Since $Z'$ is a Hilbert space valued holomorphic function in
a neighborhood of $\bar D$, by means of the Cauchy integral, the sequence $\pi_n\circ Z'$ uniformly converges to $Z'$ on $\bar D$ as $n\to\infty$.

Fix $n$ big enough such that $(\sum_{j=1}^n |Z_j'(\zeta)|^2)^{1/2} > (1-2\varepsilon)r$ for every $\zeta \in bD$. Then
$X_n:= (\pi_n \circ \Phi^{-1} \circ Z )(D) \cap (r- 2 \varepsilon)\B^n$ is a closed complex (with respect to $J_{st}$) curve through the origin in $\B^n$.
By the classical result due to Lelong (see, e.g., \cite{Ch}) we have $\Area(X_n) \geq \pi (r-2\varepsilon)^2$. Since $\Area(X_n) \le \Area(X)$ and  $\varepsilon$ is arbitrary, we obtain $r\le R$ as desired.
$\blacksquare$

\begin{cor}
\label{squizcor1}
Let $r,R > 0$ and let  $G$ be an open subset in $\Sigma_{R\D}$. Suppose that there exists a symplectomorphism
$\Phi: r\B^\infty \to G$
of class $C^1$ with respect to the $\H_0$-norm. Let $s_0>0$. Suppose that  the tangent maps $d\Phi(Z): \H_0 \to \H_0$ as well as their inverses $(d\Phi(Z))^{-1}$ are  bounded  uniformly in $Z\in r\B^\infty$. Furthermore, suppose that the antiholomorphic derivative $\Phi_{\bar Z}(Z)$ as an operator $\H_0 \to \H_{s_0}$ is bounded uniformly in $Z\in r\B^\infty$. Then $r \leq R$.
\end{cor}

The proof goes along the same lines as the proof of Theorem \ref{squiz}
because the hypotheses imply that $\tilde A$ satisfies \eqref{norm2}.

Corollary \ref{squizcor1} can be applied, for example, to symplectomorphisms of the form
$\Phi = h + K$, where $h$ is a holomorphic transformation of $\H_0$ and $K:\H_0 \to \H_{s_0}$, i.e., $K$ is a compact map $\H_0 \to \H_0$. Hence, we obtain a generalization of Kuksin's result, in which $h$ is linear. One can view Corollary \ref{squizcor1} as an infinitesimal version of theorem of Kuksin.

If  $(\H_s)$ is the Sobolev scale, the main assumption of Corollary \ref{squizcor1} means that the antiholomorphic derivative $\Phi_{\bar Z}(Z)$ is a smoothing operator. This condition is more restrictive than the assumptions of Theorem \ref{squiz} which do not require any smoothing property. In the last section, we apply Theorem \ref{squiz} to the discrete Schr\"odinger equation. The flow of the latter does not have a smoothing effect.

\section{Cauchy integral for vector functions}

We recall the modifications of the Cauchy-Green operator considered in \cite{SuTu1}. Their properties are well-known in the scalar case \cite{AIM,Mo,Ve}. The difference is that we need them for Hilbert
space-valued functions.

\subsection{ Cauchy integral and related operators}
Consider the arcs $\gamma_1 = \{ e^{i\theta} : 0 < \theta < \pi/2 \}$, $\gamma_2 = \{ e^{i\theta} : \pi/2 < \theta < \pi \}$, $\gamma_3 = \{ e^{i\theta} : \pi < \theta < 2\pi\}$ on the unit circle in $\C$.
Define the functions
\[
R(\zeta) = e^{3\pi i/4}(\zeta - 1)^{1/4} (\zeta + 1)^{1/4}(\zeta - i)^{1/2}, \qquad
X(\zeta)= R(\zeta)/\sqrt{\zeta}.
\]
Here we choose the branch of $R$ continuous in $\bar\D$
satisfying $R(0) = e^{3\pi i/4}$.
For definiteness, we also choose the branch of $\sqrt{\zeta}$ continuous in $\C$ with deleted positive real line, $\sqrt{-1}=i$. Then $\arg X$   on  arcs $\gamma_j$, $j=1,2,3$ is equal to  $3\pi/4$, $\pi/4$ and $0$ respectively.  Therefore, the function $X$ satisfies the boundary conditions
\begin{equation}
\label{BC}
\begin{cases}
\; \Im ((1+i)X(\zeta)) = 0, & \zeta \in \gamma_1,\\
\; \Im ((1-i)X(\zeta)) = 0, & \zeta \in \gamma_2,\\
\; \Im X(\zeta) = 0, &      \zeta \in \gamma_3,
\end{cases}
\end{equation}
which represent the lines through 0 parallel to the sides of
the triangle $\Delta$.

Let  $f:\D \to \C$ be a measurable function. The  Cauchy (or Cauchy-Green) operator is defined by
\[
Tf(\zeta) = \frac{1}{2\pi i} \int_\D \frac{f(t)dt \wedge d\bar{t}}{ t-\zeta}.
\]
The operator $T: L^p(\D) \to W^{1,p}(\D)$ is bounded for $p > 1$, and $(\partial/\partial\bar{\zeta}) Tf = f$ as Sobolev's derivative, i.e., $T$ solves the $\bar\partial$-problem in $\D$. Furthermore, $Tf$ is holomorphic on $\C \setminus \bar{\D}$.

Let $Q$ be a function in $\D$.  Introduce the modified Cauchy-Green  operator
\[
T_Qf(\zeta) = Q(\zeta)\left ( T(f/Q)(\zeta) + \zeta^{-1}\bar{T(f/Q)(1/\bar{\zeta})}\right ).
\]
It can be viewed as a symmetrization of the  operator $T$ with the weight $Q$.
We will need only the operators corresponding to two special weights, namely
\[
T_1f = T_Qf + 2i\Im Tf(1) \,\,\, \mbox{ with} \,\,\,Q = \zeta - 1
\]
and
\[
T_2f = T_Qf \,\,\, \mbox{ with} \,\,\,Q = R.
\]
Note that
\begin{eqnarray}
\label{OperT1}
T_1f(\zeta) = Tf(\zeta) - \bar{Tf(1/\bar{\zeta})}.
\end{eqnarray}
We also define formal derivatives
\[
S_jf(\zeta) = \frac{\partial}{\partial\zeta} T_jf(\zeta)
\]
as integrals in the sense of the Cauchy principal value. We recall the following facts \cite{SuTu1}.

\begin{prop}
\label{OpBounValScal}
The operators $T_j,S_j$ $(j=1,2)$ enjoy the following properties:
\begin{itemize}
\item[(i)] Each $S_j :L^p(\D) \to L^p(\D)$ is a bounded linear operator for $p_1 < p < p_2$. Here for $S_1$ one has $p_1 = 1$ and $p_2 = \infty$ and for $S_2$ one has $p_1 = 4/3$ and $p_2 = 8/3$. Moreover, for $f \in L^p(\D)$, one has $S_jf(\zeta) = (\partial/\partial\zeta) T_jf(\zeta)$ as Sobolev's derivative.
\item[(ii)]  Each $T_j :L^p(\D) \to W^{1,p}(\D)$ is a bounded linear operator for $p_1<p<p_2$.  Moreover, for $f \in L^p(\D)$, one has $(\partial/\partial\bar{\zeta}) T_j f = f$ on $\D$ as Sobolev's derivative.
\item[(iii)]  For every $f \in L^p(\D)$, $2 < p < p_2$, the function $T_1f$ satisfies $\Re T_1f|_{b\D} = 0$ whereas $T_2f$ satisfies the same boundary conditions (\ref{BC}) as $X$.
\item[(iv)] Each $S_j: L^2(\D) \to L^2(\D)$, $j=1,2$, is an isometry.
\item[(v)]  The function $p \mapsto \| S_j \|_{L^p}$ approaches $\| S_j \|_{L^2} = 1$ as $p\searrow 2$.
\end{itemize}
\end{prop}
We need to extend Proposition \ref{OpBounValScal} to Hilbert space-valued functions.

\subsection{Operators on spaces of vector functions}

For definiteness we only consider functions $\D \to \H$, where as usual $\H$ is a separable Hilbert space. A function $u:\D \to \H$ is called {\it simple} if it takes only a finite number of values $h_j$, $j=1,...,m$ and every preimage $u^{-1}(h_j)$ is a measurable set. The function $u$ is called {\it strongly measurable} if there exists a sequence of simple functions $(u_n)$ converging to $u$ in the norm of $\H$. A vector function is called {\it weakly measurable} if for every $h \in \H$, the function $\zeta \mapsto \langle u(\zeta), h\rangle$ is measurable. By Pettis's theorem \cite{Y} for functions with values in separable spaces these two notions coincide, so we will use the term {\it measurable}.  Note that  simple functions are dense in $L^p(\D,\H)$.

Let $P: L^p(\D) \to L^p(\D)$ be a bounded linear operator. We say that {\it $P$ extends to $L^p(\D,\H)$} if there is a unique
bounded linear operator $P_\H: L^p(\D,\H) \to L^p(\D,\H)$ such that for every $u \in L^p(\D)$ and $h \in \H$ we have
$P_\H(uh) = P(u)h$. We will usually omit the index $\H$ in $P_\H$.

\begin{prop}
\label{VectIntPropI}
\begin{itemize}
\item[(i)] Every bounded linear operator $P:L^p(\D) \to L^p(\D)$ extends to $L^p(\D,\H)$, $1 \le p < \infty$, and the extension has the same norm as $P$.
\item[(ii)] For $p > 2$ the operators $T$, $T_1$ are bounded linear operators $L^p(\D,\H) \to C^\alpha(\D,\H)$ with
$\alpha = (p-2)/p$.
\item[(iii)] For $u \in L^p(\D,\H)$ for the same $p$ as in Proposition \ref{OpBounValScal}, we have
\begin{eqnarray*}
\frac{\partial Tu}{\partial \bar \zeta} = u, \,\,\frac{\partial T_ju}{\partial \bar \zeta} = u, \,\,
\frac{\partial Tu}{\partial  \zeta} = Su, \,\,\frac{\partial T_ju}{\partial  \zeta} = S_ju , \,\, j=1,2
\end{eqnarray*}
as weak derivatives.
\item[(iv)] The operators $T$, $T_1$, $T_2$ are bounded linear operators $L^p(\D,\H) \to W^{1,p}(\D,\H)$ for the same $p$ as in Proposition  \ref{OpBounValScal}.
\end{itemize}
\end{prop}

\begin{proof} (i) If $P$ is a singular integral operator, the result follows because $\H$ is a UMD space \cite{B}. For a general bounded linear operator the result follows because $\H$ is so called $p$-space \cite{He}, which means exactly the same as Proposition \ref{VectIntPropI} (i). For completeness we give a direct proof of Proposition \ref{VectIntPropI} (i) in Appendix. Since the operators $T$, $T_1$, $T_2$, $S$, $S_1$, and $S_2$ are bounded linear operators in $L^p(\D)$ for appropriate $p > 1$, they extend to $L^p(\D,\H)$.

(ii)  Let $u = \sum_{j=1}^n \chi_k h_k: \D \to \H$ be a simple function.
Here $h_k\in \H$ and $\chi_k$ is a characteristic function of a measurable set in $\D$.
Then $T\chi_k$ are defined and $Tu$ can be defined by $Tu = \sum_{k=1}^n (T\chi_k) h_k$. The range of this  extended operator $T$ is contained in the finite dimensional space $\H_n:= \Span\{ h_k: 1 \leq k \leq n \} \subset \H$. The proof of the corresponding result for scalar functions (see for instance \cite{Ve}, Theorem 1.19) goes through with no changes for the operator $T$ extended to  $\H_n$-valued functions. Since simple functions are dense in $L^p(\D,\H)$, the proposition follows for $T$. For the operator $T_1$ the desired result follows by the formula (\ref{OperT1}).

(iii) The bounded linear form $L^1(\D) \to \C$, $u \mapsto \int_{\D} u(\zeta)d^2\zeta$ has the norm equal to 1. By (i)  it extends to $L^1(\D,\H)$ and the extension has the norm 1. This definition of the integral for $\H$-valued functions is equivalent to Bochner's integral \cite{Y}. The result (iii) follows because it holds for scalar-valued functions and because all the operators in question  are bounded in $L^p(\D,\H)$. For example, we prove that $\partial Tu/\partial  \zeta = Su$. We need to show that for every (scalar-valued) test function $\phi$
\begin{eqnarray}
\label{IntOpIII}
\int_\D \left ( (Su) \phi + (Tu)\phi \right ) d^2\zeta = 0.
\end{eqnarray}
Fix $\phi$. Since the operators $S$ and $T$ are bounded in $L^p(\D,\H)$, the left-hand part of (\ref{IntOpIII}) defines a
bounded linear form (in $u$) on $L^p(\D,\H)$. Every simple function $u$ satisfies (\ref{IntOpIII}) because  the result holds for scalar-valued functions. The conclusion now follows by density of simple functions in $L^p(\D,\H)$.

(iv) is immediate by (i) and (iii).
\end{proof}$\blacksquare$
\medskip

Although the result of (ii) for $T_2$ still holds, the argument does not go through directly. We do not need this result here.
\medskip

We need a version of Proposition \ref{VectIntPropI} (i) for two operators acting on different components of a vector-valued function.
\begin{prop}
\label{VectIntPropI-2}
Let $X_1$ and $X_2$ be Hilbert spaces and let $X=X_1\oplus X_2$
be their Hilbert direct sum. Let $P_j:L^p(\D)\to L^p(\D)$ be bounded
linear operators, $p\ge 2$, $j=1,2$.
Let $P=(P_1,P_2):L^p(\D,X)\to L^p(\D,X)$, here $P_j$ stands for its extension to $L^p(\D,X_j)$.
Then $\|P\|_{L^p(\D,X)}\le C_p m_p$, here
$C_p=2^{\frac{1}{2}-\frac{1}{p}}$, and
$m_p=\max_{j=1,2} \|P_j\|_{L^p(\D)}$.
\end{prop}
\begin{proof}
For $x=(x_1,x_2)\in X$ we introduce the $p$-norm
\[
\|x\|_p=(\|x_1\|^p+\|x_2\|^p)^{1/p}.
\]
This norm is equivalent to the Hilbert space norm
$\|x\|=\|x\|_2$ on $X$. Furthermore, for $p\ge 2$,
\[
\|x\|_p \le \|x\|_2 \le C_p \|x\|_p.
\]
Here $C_p=\sup\{\|x\|_p^{-1}: x\in X, \|x\|_2=1\}$.
Then $C_p=2^{\frac{1}{2}-\frac{1}{p}}$ because the maximum
is attained when $\|x_1\|=\|x_2\|$.

The space $X$ equipped with the $p$-norm is a Banach space,
we denote it by $X^{(p)}$. For $p\ge 2$ we immediately obtain
\[
\|f\|_{L^p(\D,X^{(p)})}\le
\|f\|_{L^p(\D,X)}\le
C_p \|f\|_{L^p(\D,X^{(p)})}.
\]
For $f=(f_1,f_2)\in L^p(\D,X)$ we have
\begin{align*}
\|Pf\|^p_{L^p(\D,X)}
&\le C_p^p \|Pf\|^p_{L^p(\D,X^{(p)})}
=C_p^p \int_\D (\|P_1 f_1(\zeta)\|^p_{X_1}
+\|P_2 f_2(\zeta)\|^p_{X_2}) d^2\zeta
\\
&\le C_p^p m_p^p \int_\D (\|f_1(\zeta)\|^p_{X_1}
+\|f_2(\zeta)\|^p_{X_2}) d^2\zeta
=C_p^p m_p^p \|f\|^p_{L^p(\D,X^{(p)})}
\le C_p^p m_p^p \|f\|^p_{L^p(\D,X)},
\end{align*}
hence the conclusion.
\end{proof}
$\blacksquare$

\section{Proof of Theorem \ref{ThDiscs}}

Let $\Phi:\D \to \Delta$ be a biholomorphism satisfying
$\Phi(\pm 1) = \pm 1$ and $\Phi(i) = i$. Note that $\Phi \in W^{1,p}(\D)$ for $p\ge 2$ close enough to $2$ by the Christoffel-Schwarz formula. We look for a solution $Z = (z,w):\D \to \H_0$  of (\ref{holomorphy}) of class $W^{1,p}(\D,\H_0)$, $p > 2$, in the form
\begin{equation}
\label{MR}
\begin{cases}
\; z = T_2u + \Phi,\\
\; w = T_1v - T_1v(\tau) + w^0.
\end{cases}
\end{equation}
for some $\tau \in \D$; hence, $w(\tau) = w^0$.

This form ensures that $z$ satisfies the desired boundary conditions, namely, takes $b\D$ to $b\Delta$. We do not have specific boundary
conditions on the $w$-components; in \eqref{MR} each component $w_j$
takes $b\D$ to a line $\Re w_j=\text{const}$. This way the $w$-components will not contribute to the area of the disc.
We could use the operator $T_2$ in both lines of \eqref{MR},
however, the choice we make here seems more natural.

The Cauchy-Riemann equation (\ref{holomorphy}) for $Z$
of the form \eqref{MR} turns into the integral equation
\begin{eqnarray}
\label{mainsystem}
\left(
\begin{array}{cl}
 u\\
 v
\end{array}
\right) = A(z,w)\left(
\begin{array}{cl}
 \bar{S_2u} +\bar{\Phi'}\\
 \bar{S_1v}
\end{array}
\right).
\end{eqnarray}
We will show  that there exists a solution of (\ref{MR}, \ref{mainsystem}) so that $z(\tau)=z^0$ for some $\tau\in\D$.

We first analyze the equation \eqref{mainsystem} for fixed
$Z\in C(\bar\D,\H_0)$.
For every $\zeta\in\D$,
the operator $A(Z(\zeta))$ is bounded in $\H_s$ and satisfies
$\|A(Z(\zeta))\|_{\H_s}\le a<1$, $s\in [0,s_0]$.
Then $A(Z)$ defines an operator on $L^p(\D,\H_s)$ such that
$\|A(Z)\|_{L^p(\D,\H_s)}\le a$.

Let $m_p=\max_{j=1,2}\|S_j\|_{L^p(\D)}$.
Then by Proposition \ref{VectIntPropI-2}, we have
$\|(S_2,S_1)\|_{L^p(\D,\H_s)}\le C_p m_p$.
Since $m_p\to 1$ and $C_p\to 1$ as $p\searrow 2$, we can fix
$p>2$ close to 2 such that $a C_p m_p<1$.

Then for every fixed $Z\in C(\bar\D,\H_0)$, by the contraction principle there exists a unique solution $U=(u,v)\in L^p(\D,\H_s)$ of the equation (\ref{mainsystem}) satisfying
\[
\|U\|_{L^p(\D,\H_s)}
\le a (C_p m_p \|U\|_{L^p(\D,\H_s)}
+ \|(\Phi',0)\|_{L^p(\D,\H_s)} ), \quad
\|U\|_{L^p(\D,\H_s)} \le M_1:=\frac{a\|\Phi'\|_{L^p(\D)}}{1-aC_p m_p}.
\]
We now obtain an {\it a priori} estimate for (\ref{MR}, \ref{mainsystem}). Indeed, by (\ref{MR}) there exists a constant $M > 0$ depending on $M_1$ and $w^0$ such that
\begin{equation}
\label{ball1}
\| Z \|_{C(\bar\D,\H_0)} \le M.
\end{equation}
We now define a continuous map $\Psi:\C \to \bar{\D}$
\[
\Psi(z) =\begin{cases}
 \Phi^{-1}(z), & z \in \bar{\Delta},\\
 \Phi^{-1}(b\Delta \cap [z^0,z]), &  z \in \C \setminus  \bar{\Delta}.
\end{cases}
\]
Here $[z^0,z]$ is the line segment from $z^0$ to $z$, and the intersection $b\Delta \cap [z^0,z]$ consists of a single point.

Let $E = B \times \bar{\D}$, here
$B= \{ Z \in C(\bar\D,\H_0): \| Z \|_{C(\bar\D,\H_0)} \le M \}$.
Introduce the map $F: E \to E$,
$F:(z,w,\tau) \mapsto (\tilde{z},\tilde{w},\tilde{\tau})$
defined by
\begin{align*}
&\tilde{z} = T_2u+\Phi,\\
&\tilde{w} = T_1v - T_1v(\tau) + w_0,\\
&\tilde{\tau} = \Psi(z^0 - T_2u(\tau)).
\end{align*}
Here $(u,v)$ is a solution of (\ref{mainsystem}).
By \eqref{ball1} the map $F$ is well defined.

The map $F$ is continuous.
Indeed, let $Z_0\in C(\bar\D,\H_0)$. Then $Z_0(\bar\D)\subset\H_0$
is compact. Then if $Z\in C(\bar\D,\H_0)$ is close to $Z_0$,
then $A(Z(\zeta))$ and $A(Z_0(\zeta))$ as operators on $\H_0$
are close in the operator norm uniformly in $\zeta\in\bar\D$.
Then the corresponding solutions $U$ and $U_0$ of \eqref{mainsystem}
are close in $L^p(\D,\H_0)$, and the continuity of $F$ follows.

By Proposition \ref{VectIntPropI}, $F(E)\subset W^{1,p}(\D,\H_s)$.
By Proposition \ref{SobolevProp} the inclusion
$W^{1,p}(\D,\H_s) \subset C(\bar\D,\H_0)$ is compact,
hence $F(E)$ is compact.
Now by Schauder's principle the continuous compact map $F$ on a convex set $E$ has a fixed point $(z,w,\tau)$. The fixed point satisfies (\ref{MR}), (\ref{mainsystem}) and $\tau = \Psi(z^0 - T_2u(\tau))$.

By (\ref{MR}) and (\ref{mainsystem}), the map $Z=(z,w)\in W^{1,p}(\D,\H_0)$, satisfies the Cauchy-Riemann equations (\ref{holomorphy}), and $w(\tau)=w^0$.

We state the rest of the conclusions of Theorem \ref{ThDiscs}
in the following

\begin{lemma}
\label{tau-lemma}
\begin{itemize}
\item[(i)] $\tau \in \D$ and $z(\tau) = z^0$.
\item[(ii)] The map $z$ satisfies $z(\bar\D) \subset \bar\Delta$,
$z(b\D) \subset b\Delta$, and $\deg z=1$;
here $\deg z$ denotes the degree of the map
$z|_{b\D}:b\D \to b\Delta$.
In particular, $Z$ satisfies (\ref{BC0}).
\item[(iii)] $\Area(Z)=1$.
\end{itemize}
\end{lemma}
\proof
We closely follow \cite{SuTu1}.
If $\tau\in b\D$, then $z^0 - T_2u(\tau)\notin\Delta$,
and the line connecting $z^0$ and $z^0 - T_2u(\tau)$
intersects a certain edge of $\Delta$. But they can not
intersect because by the boundary conditions on $T_2u$,
they are parallel. Now that $\tau\in\D$, by the definition
of $\Psi$, we have $\Phi(\tau)=z^0-T_2u(\tau)$, and $z(\tau)=z^0$.

Since $A=0$ on $\H_0\setminus\Sigma$, the function $z$ is holomorphic
at every $\zeta\in\D$ for which $z(\zeta)\notin\bar\Delta$.
Then by the maximum principle, $z(\bar\D)\subset\bar\Delta$.
By the boundary conditions on $T_2u$, the map $z$ takes the arcs
$\gamma_j$ of $b\D$ to the corresponding edges of $\Delta$,
hence $\deg z=1$.

By the structure of the standard symplectic form, all components
separately contribute to $\Area(Z)$. Since each $w_j$-component of $Z$ takes $b\D$ to a real line, it does not contribute to the area. Since $z(b\D)=b\Delta$ and $\deg z=1$, we have $\Area(Z)=\Area(\Delta)=1$ as desired. We refer to \cite{SuTu1} for more details.
$\blacksquare$
\medskip

The proof of Theorem \ref{ThDiscs} is complete.

\section{Application to the discrete Schr\"odinger equation}

Consider the following system of equations
\begin{eqnarray}
\label{Sch}
iu_n' + f(|u_n|^2)u_n + \sum_k a_{nk} u_{k}  = 0.
\end{eqnarray}
Here $u(t)=(u_n(t))_{n\in\Z}$, $u_n(t)\in \C$, $t \ge 0$.
We use the notation $u_n' = du_n/dt$.
We assume that $f: \R_+ \to \R$ and its derivative are continuous
on the positive reals, furthermore,
$\lim_{x\to 0}f(x)=\lim_{x\to 0}[xf'(x)]=0$.
For example, one can take $f(x) = x^p$ with real $p > 0$.
The hypotheses on the function $f$ are imposed in order for the flow of \eqref{Sch} to be $C^1$ smooth.
We suppose that $A = (a_{nk})$ is an infinite matrix independent of $t$. Furthermore, $A$ is a hermitian matrix, that is, $a_{nk} = \overline{a_{kn}}$. For simplicity we also assume that the entries $a_{nk}$ are uniformly bounded and there exists $m > 0$ such that $a_{nk} = 0$ if $|n - k| > m$.

The equation (\ref{Sch}) with $f(x) = x$ is called the discrete self-trapping equation \cite{ELS}. The special case with $a_{nk} = 1$ if $| n - k | = 1$ and $a_{nk} = 0$ otherwise, is the discrete nonlinear (cubic) Schr\"odinger equation:
\[
iu_n' + | u_n |^2u_n + u_{n-1} + u_{n+1}  = 0.
\]
There are other discretizations of the Schr\"odinger equation, in particular, the Ablowitz-Ladik model
\[
iu_n' + (1+|u_n|^2)(u_{n-1} + u_{n+1})  = 0.
\]
The latter does not have the form \eqref{Sch}, but it can be treated
in a similar way.
In the special case $A = 0$, the equation (\ref{Sch}) can be solved explicitly:
\[
u_n(t) = e^{i t f(| u_n(0) |^2)}u_n(0).
\]
The equation (\ref{Sch}) in this special form was suggested to us by Stephan de Bievre.

The equation (\ref{Sch}) can be written in the Hamiltonian form:
\[
u_n' = i\frac{\partial H}{\partial \overline{u_n}}.
\]
The Hamiltonian $H$ is given by
\[
H = \sum_n F(| u_n |^2) + \sum_{n,k} a_{nk} \overline{u_n}u_k,
\]
here $F' = f$ and $F(0)=0$. The equation (\ref{Sch}) preserves the $l^2(\Z)$ norm
$\|u\|_{l^2} = (\sum_n |u_n|^2)^{1/2}$.
Hence, the flow $u(0) \mapsto u(t)$ of (\ref{Sch}) is globally defined on $l^2(\Z)$ and preserves the standard symplectic form $\omega = (i/2) \sum_n du_n \wedge d\overline{u_n}$.

We claim that Theorem \ref{squiz}  applies to (\ref{Sch}), hence, the non-squeezing property holds for the flow of (\ref{Sch}). Of course, the ball in Theorem \ref{squiz} need not have a center at the origin.

Consider the standard Hilbert  scale $\H_s$, $s \in \R$, defined
in Section 2.3 using $\H_0=l^2(\Z)$ and $\theta_n = (1 + n^2)^{1/2}$.  Namely, $\H_s = \{ u=(u_n)_{n\in\Z} | \, \|u\|_s< \infty\}$, where
$\|u\|_s^2 = \sum_n |u_n|^2\theta_n^{2s}$ is the norm in $\H_s$.

Let $u(t)$ be a solution of (\ref{Sch}) such that $\|u(t)\|_0 =\|u(0)\|_0\le M$. The derivative of the flow of (\ref{Sch}) at $u(t)$ is the flow of the linear equation
\[
iv_n' + a_nv_n + b_n \overline{v_n} + \sum_k a_{nk} v_k = 0.
\]
Here
\[
a_n = f'( |u_n|^2) |u_n|^2 + f(|u_{n}|^2),\quad
b_n = f'(|u_n|^2) u_n^2.
\]
We claim that the operator $v(0) \mapsto v(t)$ is bounded in $\H_s$ for all $s \in \R$.
We have
\[
(|v_n|^2)'= 2\Re (v_n'\overline{v_n})= - 2\Im ( a_n|v_n|^2 + b_n \overline{v_n}^2 + \sum_k a_{nk} \overline{v_n} v_k).
\]
Using the assumptions on $f$, $u$, and $A$, we obtain the estimate
\[
|(|v_n|^2)'| \le C_1 (|v_n|^2 + |v_n| \sum_{|k-n| \le m} |v_k|),
\]
with $C_1$ depending on $f$, $A$, and $M$. Note that $(2/5)^{1/2} \le \theta_n/\theta_{n+1} \le (5/2)^{1/2}$ for all $n \in \Z$.  Then we have
\[
|(|v_n|^2\theta_n^{2s})'| \le C_2 (|v_n|^2\theta_n^{2s} + |v_n|
\theta_n^s \sum_{ |k-n| \le m} |v_k| \theta_k^s).
\]
Here $C_2$ depends on $C_1$ and $s$. Since
$2(|v_n|\theta_n^s)( |v_k|\theta_k^s) \le |v_n|^2\theta_n^{2s}+|v_k|^2\theta_k^{2s}$, summation by $n$ yields
\[
\sum (|v_n|^2 \theta_n^{2s})' \le C_3 \sum |v_n|^2 \theta^{2s}_n.
\]
Here $C_3 = (2m+2)C_2$. Thus,
\[
(\|v(t)\|_s^2)' \le C_3 \| v(t) \|_s^2.
\]
By the Gr\"onwall lemma
\[
\| v(t) \|_s^2 \le e^{C_3t}\| v(0) \|_s^2.
\]
Hence, the linear operator $v(0) \mapsto v(t)$ is bounded in $\H_s$, $s \in \R$ by a constant $C = e^{C_3t/2}$, depending on $t$, $s$, $\| u \|_0$ and bounds on $f$ and $A$. This completes the proof of the claim.

\section{Appendix: Proof of Proposition \ref{VectIntPropI} (i)}

We need the following
\begin{lemma}
\label{VectIntPropII}
For every  $p_0 > 1$  there is a map $\H \to L^{p_0}(0,1)$ which is an isometric embedding $\H \to L^p(0,1)$ for all $1 \le p \le p_0$.
\end{lemma}
\begin{proof} Without loss of generality assume that $p_0 = 2m$ is an even integer. Choose $A > m$. Let $\{ e_n \}_{n=1}^\infty$ be an orthonormal basis of $\H$.  For every $c = \sum_{n=1}^\infty c_ne_n$ define a ``lacunary'' Fourier series
\[
f(t) = \sum_{n=1}^\infty c_n e^{2\pi i A^nt}.
\]
It turns out  (see \cite{Gra}, proof of Theorem 3.7.4  with  $r=1$)  that $\| f \|_{2m} = \| f \|_2$.
Since the function $p \mapsto \| f \|_p$ is nondecreasing and logarithmically convex, we conclude that
$\| f \|_p =  \| f \|_2 = \| c \|_\H$ for all $1 \le p \le p_0$. Hence $c \mapsto f$ is an isometry.
\end{proof} $\blacksquare$
\medskip

Identify $\H$ with a closed subspace of $L^p(0,1)$. Then the value $u(\zeta)$ of  $u \in L^p(\D,\H)$ is a function $u(\zeta) \in L^p(0,1)$. Hence $u$ defines the map $\D \times (0,1) \to \C$,
$(\zeta,t) \mapsto u(\zeta)(t)$ which we denote by the same letter $u$. Consider a simple function $u \in L^p(\D,\H)$. Then
$u(\zeta,t) = u(\zeta)(t)= \sum_{k=1}^n \chi_k(\zeta) h_k(t)$. Here  $\chi_k$ is the characteristic function of a measurable set $E_k \subset \D$, the sets $E_k$ are disjoint and $h_k(t) \in \H \subset L^p(0,1)$. Define
\begin{eqnarray}
\label{ExtDef}
(P_\H u) (\zeta,t) = \sum_{k=1}^n (P\chi_k)(\zeta) h_k(t).
\end{eqnarray}
Obviously both $u$ and $P_\H u$ are measurable as maps $\D \times (0,1) \to \C$. Then by Fubini's theorem
\begin{align*}
\| P_\H u\|^p_p & = \int_\D \int_0^1 |(P_\H u)(\zeta,t)|^p dt\, d^2\zeta =\int_0^1 \int_\D |(P_\H u)(\zeta,t)|^pd^2\zeta \,dt\\
& \le \| P \|^p_p \int_0^1 \int_\D |u(\zeta,t)|^p d^2\zeta\, dt
= \| P \|_p^p \, \| u \|_p^p.
\end{align*}
Since simple functions are dense in $L^p(\D,\H)$, the operator $P_\H$ uniquely extends to all $L^p(\D,\H)$, and $\| P_\H \|_p = \|P \|_p$.
Note that for $u \in L^p(\D,\H)$ the image $P_\H u \in L^p(\D,\H)$ because it is true for simple functions. Furthermore, if $u = \varphi h$, $\varphi \in L^p(\D)$, $h \in \H$, then $P_\H u = (P\varphi) h$. Indeed, let $\varphi_n:\D \to \C$ be a sequence of simple functions converging to $\varphi$ in $L^p(\D)$.  Then $\varphi_n h \to \varphi h$ in $L^p(\D,\H)$
and
\[
P_\H u = P_\H (\varphi h) = \lim_{n \to \infty} P_\H (\varphi_n h) = h \lim_{n \to \infty} (P \varphi_n) = (P\varphi) h.
\]
Proposition  \ref{VectIntPropI} (i)  is proved.

\end{document}